\documentclass[a4paper,12pt]{article}
\usepackage[centertags]{amsmath}
\usepackage{amsfonts}
\usepackage{amssymb}
\usepackage{amsthm}
\usepackage{amsmath}
\usepackage{dsfont}
\usepackage{graphicx}
\usepackage{tikz, subfigure}
\usepackage{pstricks-add}
\usepackage{verbatim}

\usepackage[latin1]{inputenc}
\usepackage{tikz}
\definecolor {processblue}{cmyk}{0.96,0,0,0}
\usepackage{amsfonts,graphicx,amsmath,amssymb,hyperref,color}
\usepackage[english]{babel}
\usepackage{authblk}

\newtheorem{Theorem}{Theorem}[section]

\newtheorem{Lemma}{Lemma}[section]
\newtheorem{Corollary}{Corollary}[section]

\newtheorem{Example}{Example}[section]

\newtheorem{Remark}{Remark}[section]

\AtEndDocument{\bigskip{\footnotesize%
  \textsc{Department of Mathematics, Ningbo University, Ningbo, Zhejiang, People's Republic of China} \par
  \textit{E-mail address:} \texttt{1911071014@nbu.edu.cn} \par
}}
\AtEndDocument{\bigskip{\footnotesize%
  \textsc{Department of Mathematics, Ningbo University, Ningbo, Zhejiang, People's Republic of China} \par
  \textit{E-mail address:} \texttt{1911071003@nbu.edu.cn} \par
}}
\AtEndDocument{\bigskip{\footnotesize%
  \textsc{Department of Mathematics, Ningbo University, Ningbo, Zhejiang, People's Republic of China} \par
  \textit{E-mail address:} \texttt{1811071001@nbu.edu.cn} \par
}}
\AtEndDocument{\bigskip{\footnotesize%
  \textsc{Department of Mathematics, Ningbo University, Ningbo, Zhejiang, People's Republic of China} \par
  \textit{E-mail address:} \texttt{zhaobing@nbu.edu.cn} \par
}}
\AtEndDocument{\bigskip{\footnotesize%
  \textsc{Department of Mathematics, Ningbo University, Ningbo, Zhejiang, People's Republic of China} \par
  \textit{E-mail address:} \texttt{kanjiangbunnik@yahoo.com, jiangkan@nbu.edu.cn} \par
}}

\usepackage{lipsum}% http://ctan.org/pkg/lipsum

\makeatletter
\newcommand*{\rom}[1]{\expandafter\@slowromancap\romannumeral #1@}
\makeatother
\begin{document}
\title{xxxx}
\date{}
 \title{On continuous images of  self-similar sets }
 \author{Yuanyuan Li, Jiaqi Fan, Jiangwen Gu, Bing Zhao, Kan Jiang\thanks{Kan Jiang is the corresponding author}}
\maketitle{}
\begin{abstract}
Let $(\mathcal{M}, c_k, n_k,\kappa)$ be a class of homogeneous Moran sets.
Suppose $f(x,y)\in C^3$ is a function defined on $\mathbb{R}^2$. Given $E_1, E_2\in(\mathcal{M}, c_k, n_k,\kappa) $, 
 in this paper, we prove,  under some checkable conditions on the partial  derivatives of $f(x,y)$,  that
$$f(E_1,E_2)=\{f(x,y):x\in E_1,y\in E_2\}$$ is exactly a closed interval or a union of finitely many closed intervals.
Similar results for the homogeneous self-similar sets  with arbitrary overlaps can be obtained. Further generalization is available for some inhomogeneous self-similar sets if we utilize the approximation theorem.
\end{abstract}

\section{Introduction}
Let $f(x,y)$ be a continuous function defined on $\mathbb{R}^2$, and $K$ be a self-similar set  \cite{Hutchinson} defined on $\mathbb{R}$. Denote
$$f(K,K)=\{f(x,y):x,y\in K\}.$$
We call $f(K,K)$ the continuous image of  $K$. It is natural to ask what is the exact form of $f(K,K)$. The first result in this direction is due to Steinhaus \cite{HS} who proved in 1917 the following results:
$$C+C=\{x+y:x,y\in C\}=[0,2], C-C=\{x-y:x,y\in C\}=[-1,1],$$ where $C$ is the middle-third Cantor set.
In 2019, Athreya, Reznick and Tyson \cite{Tyson} proved that
\[
C\div C=\left\{\dfrac{x}{y}:x,y\in C, y\neq0\right\}=\bigcup_{n=-\infty}^{\infty}\left[ 3^{-n}\dfrac{2}{3},3^{-n}\dfrac
{3}{2}\right] \cup\{0\}.
\]
In \cite{Gu}, Gu, Jiang, Xi and Zhao gave the topological structure of $$C\cdot C=\{xy:x,y\in C\}.$$ Note that $C\cdot C$ is not a closed interval. In \cite{Kan2019}, Tian et al. analyzed a class of self-similar sets with overlaps, and gave a necessary and sufficient condition under which the algebraic product of two self-similar sets is exactly a closed interval. However, for a general function $f(x,y)$ and $K$, to calculate the exact form of  $f(K,K)$ is a very difficult problem.  To the best of our knowledge, up to date, there are very few sufficient conditions for a general $f$ under which the continuous image of  $K$ is  a closed interval.

In this paper, we  shall consider homogeneous  Moran sets,  under some conditions,  such that the continuous image of the Moran sets  is exactly a closed interval or a union of finitely many closed intervals. Similar results can be proved for the homogeneous self-similar sets.

We now introduce how to construct a Moran set.
 Let $E=[0,1]$.  Given two sequences $\{c_k\}_{k=1}^{\infty}$ and $\{n_k\}_{k=1}^{\infty}$ with the property $c_kn_k<1$ for any $k\geq 1$, where $c_k\in (0,1/2)$ and $n_k\in \mathbb{N}_{\geq 2}$,  $k=1,2,3,\cdots$.
In the first level, we choose $n_1$ sub closed interval of $[0,1]$, denoted from left to right by
$$I^{(1)}_{1},I^{(1)}_{2}, \cdots I^{(1)}_{n_1}, $$
 such that
 \begin{itemize}
  \item [(1)]  $\cup_{i=1}^{n_1}I^{(1)}_{i}\subsetneq E$;
 \item [(2)] $\dfrac{|I_i^{(1)}|}{|E|}=c_1, 1\leq i\leq n_1$,  where $|A|$ denotes the length of $A$;
  \item [(3)] $$\min_{1\leq i\leq n_1}\{x: x\in I^{(1)}_{i}\}=\min\{x:x\in I^{(1)}_{1} \}=\min\{x:x\in E\}=0,$$ and  $$\max_{1\leq i\leq n_1}\{x: x\in I^{(1)}_{i}\}=\max\{x:x\in I^{(1)}_{n_1} \}=\max\{x:x\in E\}=1;$$
     \item [(4)] given $I^{(1)}_{i}$ and $I^{(1)}_{j}$ with $i<j$, then
   $$\dfrac{|I^{(1)}_{i}\cap I^{(1)}_{j}|}{|I^{(1)}_{i}|}\leq \kappa, 0\leq \kappa < 1.$$
 \end{itemize}
 For simplicity, we call   $I^{(1)}_{i}, 1\leq i\leq n_1$,  a basic interval with length $c_1$.

 Generally, for $k\geq 1$, given a basic interval  $I_{i}^{(k)}$, then
we choose $n_{k+1}$ sub closed interval of $I_{i}^{(k)}$, denoted from left to right by
$$I^{(k+1)}_{1},I^{(k+1)}_{2}, \cdots, I^{(k+1)}_{n_{k+1}}, $$
 such that
 \begin{itemize}
  \item [(1)]  $\cup_{i=1}^{n_{k+1}}I^{(k+1)}_{i}\subsetneq I_{i}^{(k)};$
 \item [(2)] $\dfrac{|I_j^{(k+1)}|}{|I_i^{(k)}|}=c_{k+1}, 1\leq j\leq n_{k+1}$;
  \item [(3)] $$\min_{1\leq i\leq n_{k+1}}\{x: x\in I^{(k+1)}_{i}\}=\min\{x: x\in I^{(k+1)}_{1}\}=\min\{ x:x\in I^{(k)}_{i}\},$$ and  $$\max_{1\leq i\leq n_{k+1}}\{x: x\in I^{(k+1)}_{i}\}=\max\{x: x\in I^{(k+1)}_{n_{k+1}}\}=\max\{ x:x\in I^{(k)}_{i}\};$$
   \item [(4)] given two $I^{(k+1)}_{i}$ and $I^{(k+1)}_{j}$ with $i<j$, then
   $$\dfrac{|I^{(k+1)}_{i}\cap I^{(k+1)}_{j}|}{|I^{(k+1)}_{i}|}\leq \kappa, 0\leq \kappa<1.$$
 \end{itemize}
 We say the above four conditions are the Moran conditions. We call each $I^{(k)}_{i}$ a $k$-th level basic interval with length $c_1c_2\cdots c_k$.
 Clearly, for the given sequences $\{c_k\}$ and $\{n_k\}$,  we may find  many classes of different basic intervals for each $k$-th level satisfying the  Moran conditions, i.e. for a given $k\ge1$, the union of all the basic intervals in the $k$-th level may differ as by the Moran conditions, we may adjust the locations of basic intervals.

 For each $k\geq 1$, if we can find some basic intervals satisfying the Moran conditions, denoted  their union by $$C_k=\bigcup_{i=1}^{n_1\cdots n_k}I^{(k)}_{i},$$ then we call
 $$M=\bigcap_{k=1}^{\infty}C_k$$ a Moran set generated by $\{c_k\}$  and $\{n_k\}$. We denote by $(\mathcal{M}, c_k,n_k, \kappa)$ all the  Moran sets generated by $\{c_k\}$ and  $\{n_k\}$.
 \begin{Remark}
 We have the following remarks for the definition of  the Moran sets.
 \begin{itemize}
 \item[(1)]
 By the third  Moran condition, we conclude that the convex hull of each Moran set is $[0,1]$.
  \item[(2)] By the Moran conditions, it is important to point out that the definition of our Moran sets has essential differences comparing with the  classical definition.  As we allow some ``overlaps", i.e. for  the classical definition of Moran set \cite{Wen1997}, we should assume that  any two  basic intervals in $k$-th level intersect at most one point. We,  however, allow serious  overlaps (intersections).

   \item[(3)] An important class of Moran sets is  the so-called homogeneous self-similar sets, i.e. for the attractor $K$ with  the following IFS,
   $$\{f_i(x)=\lambda x+a_i\}_{i=1}^{n},$$ where $ a_i\in \mathbb{R}, 0<\lambda<1/2.$ We assume that the  convex hull of $K$ is $[0,1]$, and $K\neq [0,1]$.
   For this case each $c_k=\lambda$ and $n_k=n$.
    \item[(4)] Without the last Moran condition, some basic intervals can coincide. To avoid this case, we  let $\kappa<1$. If $\kappa=0$, then the intersection of two basic intervals is an empty set or only one point. 
  \end{itemize}
 \end{Remark}
 Before we state the main result, we need some notation which can simplify the  conditions in the following theorems.
 For any real numbers $\ell_1,\ell_2$ and $f(x,y)\in C^3$, we  denote
 $$(\ell_1\partial_x+\ell_2\partial_y)^2f(x,y)=\ell_1^2\partial_{xx}f(x,y)+2\ell_1\ell_2\partial_{xy}f(x,y)+\ell_2^2\partial_{yy}f(x,y).$$
 Let $E_1,E_2\in(\mathcal{M}, c_k, n_k,\kappa)$.
Denote by  $C_k$ and $D_k$  the unions of  basic intervals in the $k$-th level for $E_1$ and $E_2$, respectively, where $k$ is some given positive integer. $C_k$ and $D_k$ have $n_1\cdots n_k$ closed intervals. Each interval has length $c_1c_2\cdots c_k.$
For a basic interval in the $k$-th level, denoted by $I_i^{(k)}$, we define
$$\widetilde{I_i^{(k)}}=\bigcup_{j=1}^{n_{k+1}}I_{j}^{(k+1)}.$$
Namely, $\widetilde{I_i^{(k)}}$ is a union of all the $(k+1)$-th level basic intervals contained in $I_i^{(k)}.$ We define the following notation which is motivated by the  Kronecker delta,
\begin{equation*}
   \delta_{xy}=\left\lbrace\begin{array}{cc}
                1,   \partial_x f(x,y)>0, \partial_y f(x,y)>0 \mbox{ for any }(x,y)\in [0,1]^2\\
                 1,   \partial_x f(x,y)<0, \partial_y f(x,y)<0 \mbox{ for any }(x,y)\in [0,1]^2\\
                  -1,   \partial_x f(x,y)<0, \partial_y f(x,y)>0 \mbox{ for any }(x,y)\in [0,1]^2\\
                   -1,   \partial_x f(x,y)>0, \partial_y f(x,y)<0 \mbox{ for any }(x,y)\in [0,1]^2.
                \end{array}\right.
\end{equation*}
Throughout the paper (in Theorems \ref{Main1} and \ref{SSS1}), we always assume that for any $(x,y)\in [0,1]^2$, for the partial derivatives $\partial_x f(x,y),  \partial_y f(x,y)$, there are only above four cases.
For any $k\geq 1$, 
define $\xi_k=c_k(2+(n_k-2)(1-\kappa))$.
 Now, we give the main results of this paper.
 \setcounter{Theorem}{1}
  \begin{Theorem}\label{Main1}
Let $E_1,E_2\in(\mathcal{M}, c_k, n_k,\kappa)$. Suppose that $f(x,y)\in C^3$. If for any $$(x,y)\in [0,1]\times [0,1]$$ and any $k\geq 1$, we have
\begin{equation*}
   \left\lbrace\begin{array}{cc}
   \partial_{x} f(x,y)\neq 0,  \partial_{y} f(x,y)\neq 0,\\
               (c_k\partial_{x}+ \delta_{xy}(\xi_k-1)\partial_{y})^2f(x,y)\geq 0\\
(\partial_{y}+ \delta_{xy}(\xi_k-1)\partial_{x})^2f(x,y)\geq 0\\
          1-\xi_k  \leq   \delta_{xy}\dfrac{\partial_{y}f(x,y)}{\partial_{x} f(x,y)}\leq \dfrac{c_k}{1-\xi_k}, 
                \end{array}\right.
\end{equation*}
then $f(E_1, E_2)$ is a closed interval.
Moreover, given some $p\geq 1$, 
 if for any $$(x,y)\in C_p\times D_p$$ and any $k\geq p$, we have the above inequalities,
then $f(E_1, E_2)$ is a union of finitely many closed intervals.
\end{Theorem}
We compare our result with the classical Newhouse's thickness theorem \cite{Palis}.   Roughly speaking, our result can be viewed as the non-linear version of the Newhouse's thickness theorem as the conditions for the partial derivatives are similar to the definition of thickness (in particular, the ratio of first  order partial derivatives is very similar to the thickness).  
The Newhouse's thickness theorem needs to calculate the thickness while for our result we should analyze the partial derivatives.  
 The Moran sets are ``random" as the locations of basic intervals can be changed under the Moran conditions.  Therefore, it is not easy to calculate the thickness of a Moran set. 
For further comparison,
we  let $f(x,y)=x+sy$ with  a parameter $s\in \mathbb{R}$. Note that for  this case  $$\partial_{xx}f=0, \partial_{xy}f=0, \partial_{yy}f=0.$$ Here we  observe that for the above linear case, we cannot use any information of the second order partial derivatives.  
 However, we can still obtain partial result.
 \setcounter{Corollary}{2}
\begin{Corollary}
Let $E_1,E_2\in(\mathcal{M}, c_k, n_k,\kappa)$. Suppose that $f(x,y)=x+sy, s\neq 0$. If for any $$(x,y)\in [0,1]\times [0,1],$$ we have
\begin{equation*}
          \sup_{k}\left\{1-\xi_k \right\} \leq  | s|\leq \inf_{k}\left\{\dfrac{c_k}{1-\xi_k}   \right\},
\end{equation*}
then $f(E_1, E_2)$ is a closed interval. In particular, when  $s=1$ or $-1$, if for any $k\ge 1$, we have 
$$c_k+\xi_k\geq 1,$$ then
$$E_1+E_2=[0,2], E_1-E_2=[-1,1].$$  
\end{Corollary}
\setcounter{Remark}{3}
\begin{Remark}
The conditions in Theorem \ref{Main1} are  a little strong. This is natural as  for the Moran sets, in certain sense, are ``random". That means the locations of basic intervals can be changed. However, for the self-similar sets, once the IFS's are given, the locations of basic intervals are determined.  Comparing with the Moran sets,  the conditions, which guarantee the continuous image of self-similar sets is a closed interval, can be mildly weakened. In particular, the conditions for the  constant $\kappa$ can be weakened. 
\end{Remark}
Now we introduce the main result of  the homogeneous self-similar sets. We need  some basic definitions.
Let  $K_1$ and $K_2$ be the attractors of the   IFS's $$\{f_i(x)=\lambda x+a_i\}_{i=1}^{n}, \{g_j(x)=\lambda x+b_j\}_{j=1}^{m},$$ respectively, where $ a_i,b_j\in \mathbb{R}, 0<\lambda<1.$
Without loss of generality, we may assume the convex hull of $K_1$ and $K_2$ is $[0,1]$, and $ f_1(0)=0,  f_n(1)=1, g_1(0)=0,  g_m(1)=1$, and  $$f_i(0)\leq  f_{i+1}(0),1\leq  i\leq n-1, g_j(0)\leq  g_{j+1}(0),1\leq  j\leq m-1.$$
     Let $$F_1=\{1\leq i\leq n-1:f_i(1)-f_{i+1}(0)<0\}, \tau_1=\max_{i\in F_1}\{f_{i+1}(0)-f_i(1)\}.$$
     $$F_2=\{1\leq j\leq m-1:g_j(1)-g_{j+1}(0)<0\}, \tau_2=\max_{j\in F_2}\{g_{j+1}(0)-g_j(1)\}.$$
     Note that if $F_i=\emptyset$, then $K_i$ is an interval. To avoid this trivial case, we assume that $F_i\neq \emptyset, i=1,2.$
     For any $(i_1\cdots i_k)\in \{1,2,\cdots, n\}^k, (j_1\cdots j_k)\in \{1,2,\cdots, m\}^k$, we call $f_{i_1\cdots i_k}([0,1])$ and $g_{j_1\cdots j_k}([0,1])$  basic intervals with respect to $K_1$ and $K_2$. Note that the length of $I=f_{i_1\cdots i_k}([0,1])$ and $J=g_{j_1\cdots j_k}([0,1])$ is $\lambda^k$, we define 
     $$\widetilde{I}=\cup^{n}_{i=1}f_{i_1\cdots i_ki}([0,1]), \widetilde{J}=\cup^{m}_{j=1}g_{j_1\cdots j_kj}([0,1]).$$
      We still use $C_k$ and $D_k$ to denote all the basic intervals of $K_1$ and  $K_2$ in the $k$-th level.    
Now we state the result  for the homogeneous self-similar sets.
\setcounter{Theorem}{4}
\begin{Theorem}\label{SSS1}
Let $K_1,K_2$ be the attractors  defined as above.
Suppose $f(x,y)\in C^3$. If for any $(x,y)\in [0,1]\times [0,1], 1\leq l\leq m-1, 1\leq j\leq n-1$, we have
\begin{equation*}
   \left\lbrace\begin{array}{cc}
                 \partial_x f(x,y)\neq 0, \partial_y f(x,y)\neq 0\\
                (\lambda\delta_{xy}\partial_{x}+(g_l(1)-g_{l+1}(0))\partial_{y})^2f(x,y)\geq 0\\
              ( \delta_{xy}(f_j(1)-f_{j+1}(0))\partial_{x}+\partial_{y})^2 f(x,y)\geq 0\\
                \tau_1\leq\delta_{xy}\dfrac{\partial_y f(x,y)}{\partial_x f(x,y)}\leq\dfrac{\lambda}{\tau_2}
                                \end{array}\right.
\end{equation*}
then $f(K_1,K_2)$ is a closed interval. Moreover, given $p\geq 1$, 
 if for any $$(x,y)\in C_p\times D_p,$$ we have the above inequalities,
then $f(K_1, K_2)$ is a union of finitely many closed intervals.\end{Theorem}
For some inhomogeneous self-similar sets, we may still use the above result. As for a general self-similar set, it can be approximated by a sub  homogeneous  self-similar set in the sense of Hausdorff dimension \cite{PS}. Therefore, we have the following result.
\setcounter{Corollary}{5}
\begin{Corollary}
 Given $f(x,y)\in C^3$.
Let $K\subset \mathbb{R}$ be any self-similar set with positive ratios. If there exists a sub homogeneous self-similar set, denoted by $K^{\prime}$, such that
 \begin{itemize}
 \item[(1)] $conv(K^{\prime})=conv(K)=[0,1]$;
 \item[(2)] for any $(x,y)\in conv(K)\times conv(K)$, $f(x,y)$ satisfies the conditions in the above theorem with respect to the IFS of $K^{\prime}$, where $conv(.)$ denotes the convex hull;
 \end{itemize}
 then $f(K,K)$ is a closed interval.
\end{Corollary}
For the middle-third Cantor set, we have the following results.
\begin{Corollary}\label{Cantor1}
Let $C$ be the attractor of the following IFS
$$\left\{f_1(x)=\dfrac{x}{3},f_2(x)=\dfrac{x+2}{3} \right\}.$$
Suppose $f(x,y)\in C^3$ is a function defined on $\mathbb{R}^2$.
If for any $$(x,y)\in [0,1]\times [0,1],$$ we have
\begin{equation*}
  \left\lbrace\begin{array}{cc}
                \partial_x f(x,y)>0, \partial_y f(x,y)>0\\
                  \partial_{xx} f(x,y)-6\partial_{xy}f(x,y)+ 9\partial_{yy} f(x,y)\geq 0\\
               \partial_{xx} f(x,y)-2\partial_{xy}f(x,y)+ \partial_{yy} f(x,y)\geq 0\\
                  \dfrac{1}{3}\leq  \dfrac{\partial_y f(x,y)}{\partial_x f(x,y)}\leq 1, 
                \end{array}\right.
\end{equation*}
then  $f(C,C)$ is a closed interval. 
\end{Corollary}
Note that  the second  and third conditions in Corollary \ref{Cantor1} can be easily checked if we replace them by  the following conditions. 
\begin{Corollary}\label{Cantor2}
Let $C$ be the attractor of the following IFS
$$\left\{f_1(x)=\dfrac{x}{3},f_2(x)=\dfrac{x+2}{3} \right\}.$$
Suppose $f(x,y)\in C^3$ is a function defined on $\mathbb{R}^2$.
If for any $$(x,y)\in [0,1]\times [0,1],$$ we have

\begin{equation*}
  \left\lbrace\begin{array}{cc}
               \partial_x f(x,y)>0, \partial_y f(x,y)>0\\
                  \partial_{xx} f(x,y)\geq 0, \partial_{xy}f(x,y)\leq 0, \partial_{yy} f(x,y)\geq 0\\
                  \dfrac{1}{3}\leq\dfrac{\partial_y f(x,y)}{\partial_x f(x,y)}\leq1,
                \end{array}\right.
\end{equation*}
then  $f(C,C)$ is a closed interval. 
\end{Corollary}
In terms of  Corollaries  \ref{Cantor1} and \ref{Cantor2}, it is easy to prove the following two results.
Let $K$ be the attractor of the IFS
$$\left\{f_1(x)=\dfrac{x}{3},f_2(x)=\dfrac{x}{4},f_3(x)=\dfrac{x+2}{3}\right\}.$$
Set 
$$H=f([0,1],[0,1])=\left[\min_{(x,y)\in [0,1]^2}f(x,y),\max_{(x,y)\in [0,1]^2}f(x,y)\right].$$
\setcounter{Example}{8}
\begin{Example}
Let $f(x,y)=x^2+y^2+\alpha x+\beta y+\gamma xy$, where
\begin{equation*}
\left\lbrace\begin{array}{cc}
 2+\alpha+\gamma<3\beta\\
 \alpha>2+\beta+\gamma\\
 \alpha, \beta>0\\
 0<\gamma<2.
 \end{array}\right.
\end{equation*}
Then
 $$ f(K,K)=H.$$
By Corollary \ref{Cantor1}, we have $$f(C,C)=H.$$  Therefore, $$H=f([0,1], [0,1])\supset f(K,K)\supset f(C,C)=H.$$
\end{Example}
Similarly, by Corollary \ref{Cantor2} we obtain the following example.
\begin{Example}
Let $f(x,y)=\sin(\alpha xy)+\beta x+\gamma y$, where
\begin{equation*}
\left\lbrace\begin{array}{cc}
 -1<\alpha<0\\
10<\beta<\infty\\
\gamma=\dfrac{\beta}{2}.
 \end{array}\right.
\end{equation*}
Then
 $$ f(K,K)=H.$$
\end{Example}
This paper is arranged as follows. In section 2 and 3 we give the proofs of Theorems \ref{Main1} and \ref{SSS1}. In section 4, we give some remarks. 
\section{Proof of Theorem \ref{Main1}}
Let $E_1, E_2$ be two Moran sets from the class $(\mathcal{M}, c_k,n_k, \kappa).$ Namely, there exist some $C_k$ $D_k$, i.e. the union of all the basic intervals in the $k$-level, such that 
$$E_1=\cap_{k=1}^{\infty}C_k, E_2=\cap_{k=1}^{\infty}D_k.$$
The following lemma is obvious.
 \begin{Lemma}\label{key1}
Let  $f:\mathbb{R}^2\to \mathbb{R}$ be a continuous function. Then
$$f(E_1,E_2)=\bigcap_{k=1}^{\infty}f(C_k,D_k).$$
\end{Lemma}
\begin{proof}
The proof is due to two facts, $f$ is continuous and $C_k\supset C_{k+1}, D_k\supset D_{k+1}$ for any $k\ge1.$
\end{proof}
To prove $f(E_1,E_2)$ is an interval, it  suffices to prove that
\begin{itemize}
\item [(1)] for any $k\geq 1$, we have  $f(C_k,D_k)=f(C_{k+1},D_{k+1});$
\item [(2)]  $f(C_1, D_1)$ is an interval.
\end{itemize}
To guarantee $(1)$ and $(2)$, we have the following sufficient condition.
\begin{Lemma}\label{key2}
Let  $k\geq 0$. Suppose  $I$ and $J$ are  any two basic intervals in $C_k$ and $D_k$ respectively.
If
 $$f(I, J)=f(\widetilde{I}, \widetilde{J}),$$
 then
 $f(C_k,D_k)=f(C_{k+1},D_{k+1}).$ Here if $k=0$,  then we assume  $$I=J=C_0=D_0=[0,1].$$ 
\end{Lemma}
\begin{proof}
Let $$C_k=\cup_{i=1}^{t_k}I_{i}^{(k)}, D_k=\cup_{j=1}^{t_k}J_{j}^{(k)}, t_k=n_1n_2\cdots n_k, k\geq 1. $$
Then we have
\begin{eqnarray*}
f(C_k, D_k)&=& \cup_{1\leq i\leq t_k} \cup_{1\leq j\leq t_k}f(I_{i}^{(k)}, J_{j}^{(k)})\\&=
& \cup_{1\leq i\leq t_k} \cup_{1\leq j\leq t_k}f(\widetilde{I_{i}^{(k)}},\widetilde{J_{j}^{(k)}})\\&=&
f(\cup_{1\leq i\leq t_k}\widetilde{I_{i}^{(k)}},\cup_{1\leq j\leq t_k}\widetilde{J_{j}^{(k)}})\\&=&f(C_{k+1}, D_{k+1}).
\end{eqnarray*}
\end{proof}
\begin{proof}[\bf{Proof of Theorem \ref{Main1}}]
Theorem \ref{Main1} contains four cases, i.e.  for any $(x,y)\in [0,1]^2$, we have 
\begin{equation*}
\left\lbrace\begin{array}{cc}
\partial_x f(x,y)>0, \partial_y f(x,y)>0\\
\partial_x f(x,y)<0, \partial_y f(x,y)<0\\
\partial_x f(x,y)<0, \partial_y f(x,y)>0\\
\partial_x f(x,y)>0, \partial_y f(x,y)<0.
 \end{array}\right.
\end{equation*}
For simplicity, we only prove the first case. For the other three cases, we give the outline of the proof.  The proof  of $f(E_1, E_2)$ is exactly  finitely many closed intervals is simiar to the previous case, we omit the details for this case.

Now we suppose $$\partial_x f(x,y)>0, \partial_y f(x,y)>0 \mbox{ for any }(x,y)\in [0,1]^2.$$
 Let $I=[a,a+t], J=[b,b+t]$ be any two basic intervals from $C_{k-1}$ and $D_{k-1}$, respectively, i.e. $t=|I|=|J|=c_1c_2\cdots c_{k-1}$.   If $k=1$, we let $$I=J=[0,1].$$ By Lemmas \ref{key1} and  \ref{key2}, if we can   show $$f(I, J)=f(\widetilde{I}, \widetilde{J}),$$ then we finish the proof of  Theorem \ref{Main1} for the case $\partial_x f(x,y)>0, \partial_y f(x,y)>0$. 
 
 By the definitions of $\widetilde{I}$ and $\widetilde{J}$, we let
 $$\widetilde{I}=\cup_{i=1}^{n_{k}}I_{i}^{(k)}, \widetilde{J}=\cup_{j=1}^{n_{k}}J_{j}^{(k)}.$$ For simplicity, we assume that the above intervals are arranged from left to right in terms of  the Moran  construction.
To prove $$f(I, J)=f(\widetilde{I}, \widetilde{J}),$$ it suffices to prove
the following:
\begin{itemize}
\item [(1)] for any $1\leq i\leq n_{k}$, $f(I_{i}^{(k)}, \cup_{j=1}^{n_{k}}J_{j}^{(k)})$ is a closed interval;
\item [(2)]
$$f(I_{i}^{(k)}, \cup_{j=1}^{n_{k}}J_{j}^{(k)}) \cup f(I_{i+1}^{(k)}, \cup_{j=1}^{n_{k}}J_{j}^{(k)}),$$ is a closed interval for any $1\leq i\leq n_{k}-1.$
\end{itemize}
\begin{figure}[tbph]
\centering\includegraphics[width=0.4\textwidth]{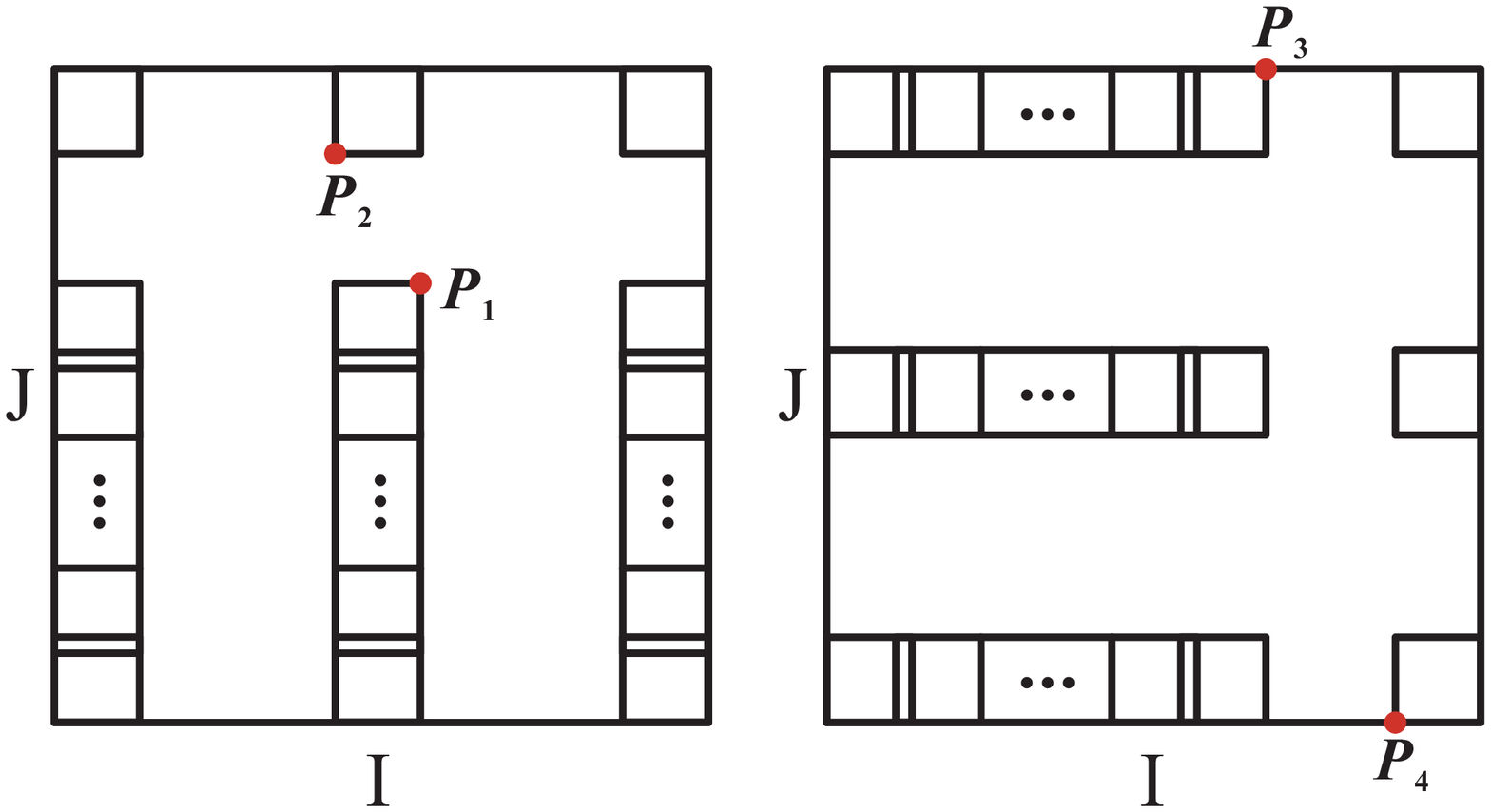} \vspace{0cm}
\centering\includegraphics[width=0.4\textwidth]{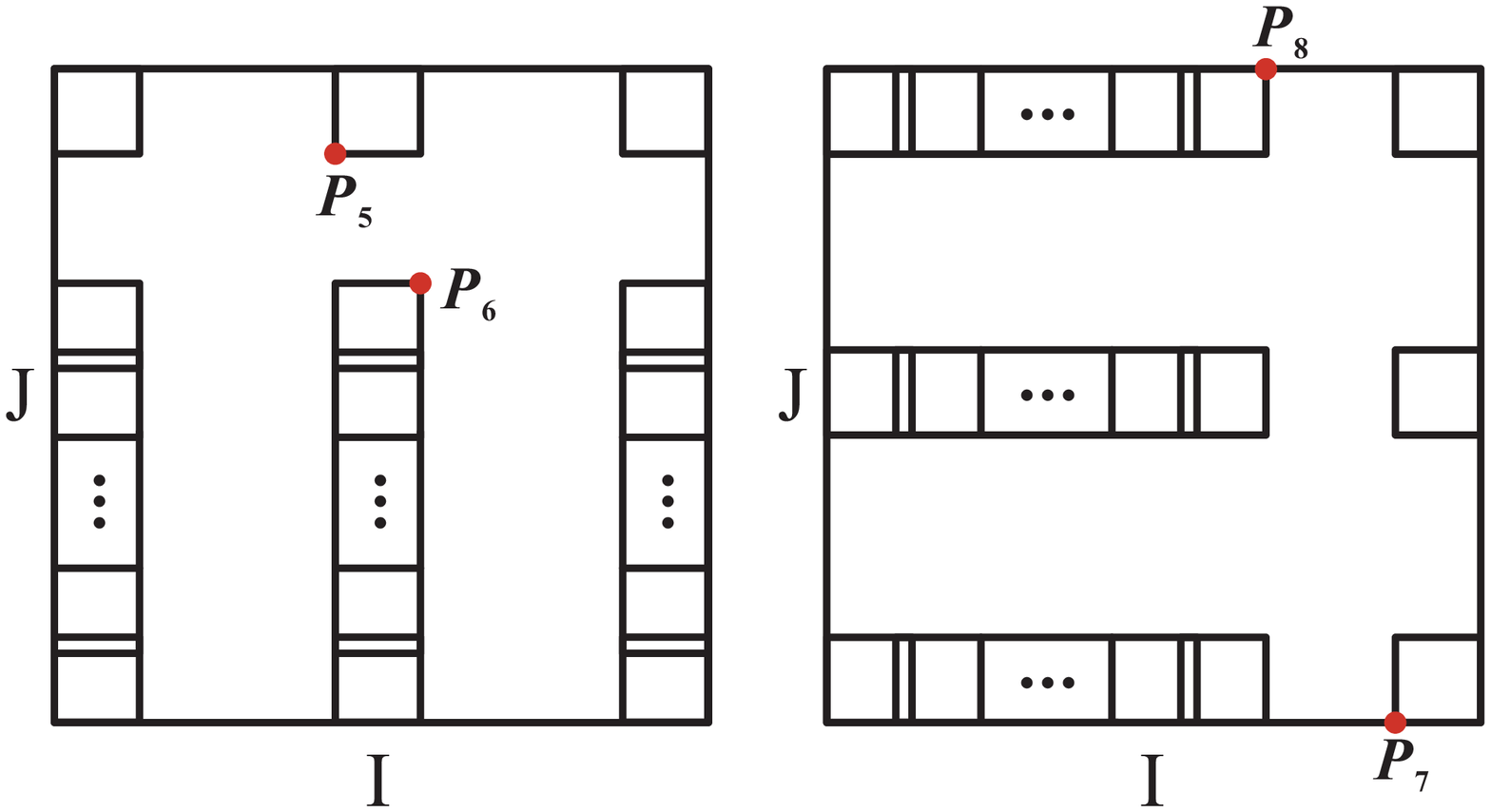} \vspace{0cm}

\caption{$\partial_x f>0, \partial_y f>0$ and $\partial_x f<0, \partial_y f<0$}
\end{figure}
\begin{figure}[tbph]
\centering\includegraphics[width=0.4\textwidth]{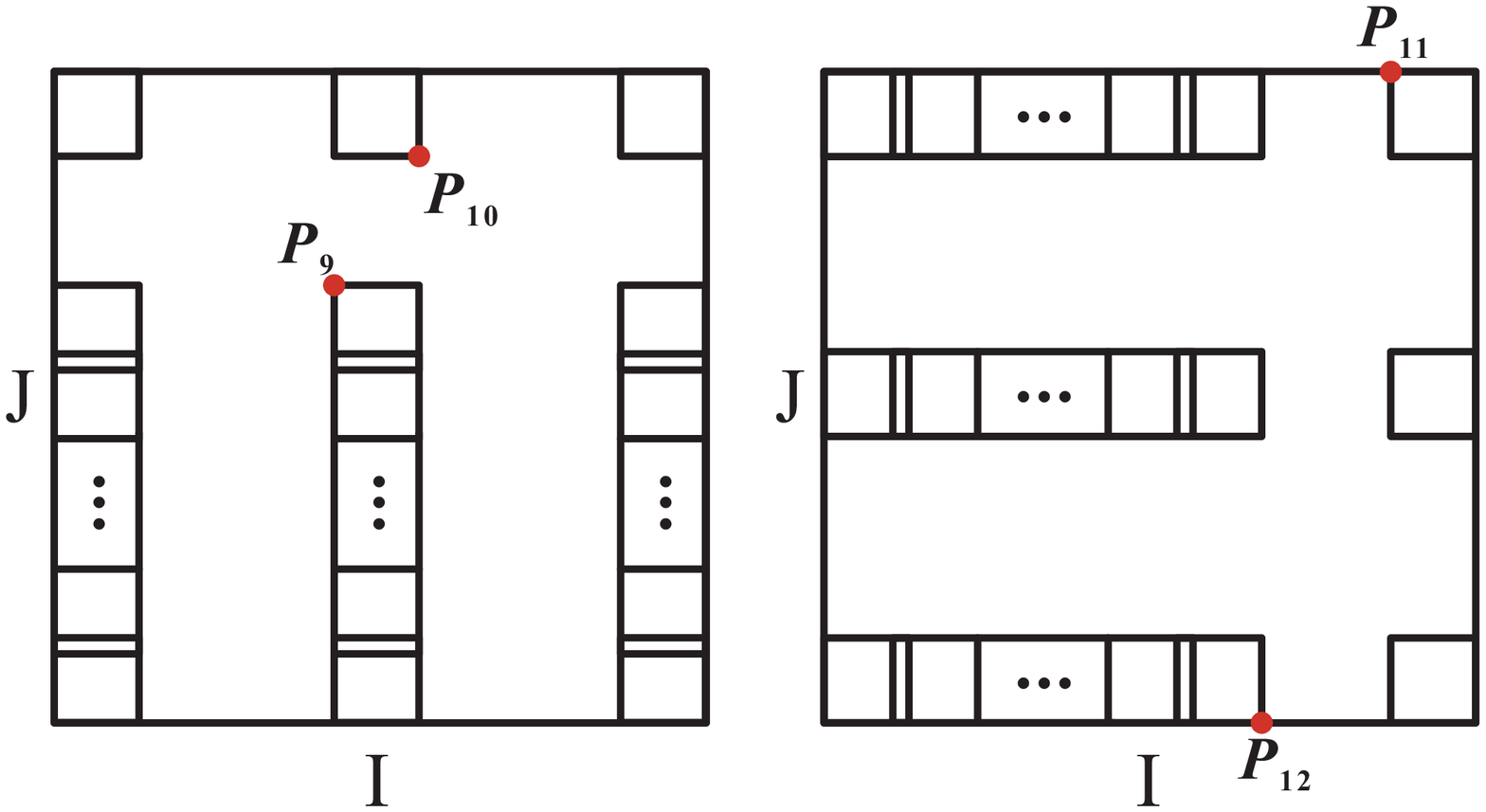} \vspace{0cm}
\centering\includegraphics[width=0.4\textwidth]{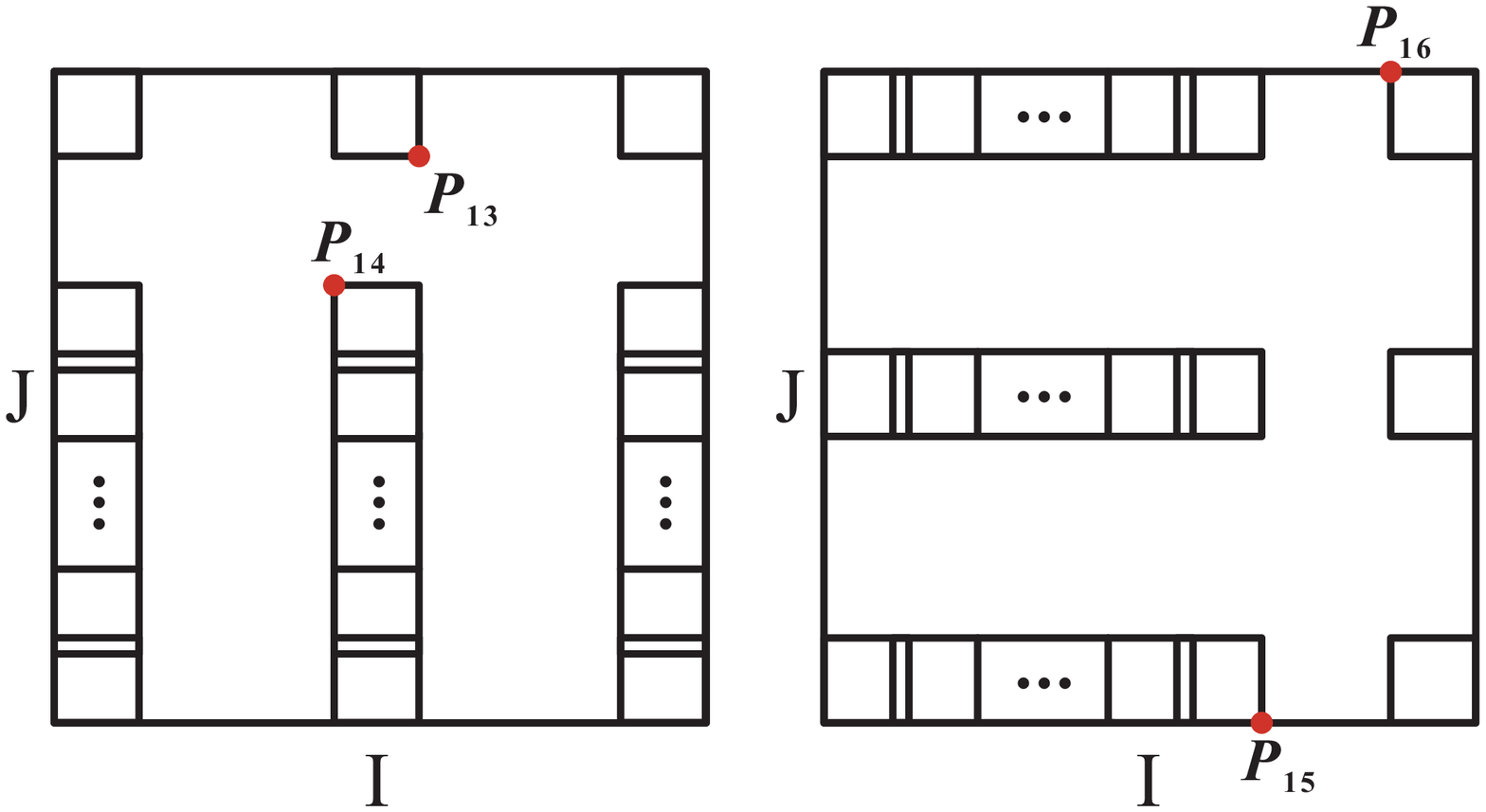} \vspace{0cm}
\caption{$\partial_x f<0, \partial_y f>0$ and $\partial_x f>0, \partial_y f<0$}
\end{figure}

 To prove
 $$f(I_{i}^{(k)}, \cup_{j=1}^{n_{k}}J_{j}^{(k)})$$ is a closed interval, it suffices to prove that
 for any $1\leq j\leq n_k-1 $, $$f(I_{i}^{(k)}, J_{j}^{(k)})\cup f(I_{i}^{(k)}, J_{j+1}^{(k)})$$ is an interval.
 By the assumptions  $$\partial_x f(x,y)>0, \partial_y f(x,y)>0$$ for any $(x,y)\in [0,1]\times [0,1]$,
 we only need to show that the right endpoints of $f(I_{i}^{(k)}, J_{j}^{(k)})$ is not less than the left endpoint of $f(I_{i}^{(k)}, J_{j+1}^{(k)})$ (if  the squares  $I_{i}^{(k)}\times J_{j}^{(k)}$ and $I_{i}^{(k)}\times J_{j+1}^{(k)}$ are intersected, then we do not need to prove this as clearly in this case $f(I_{i}^{(k)}, J_{j}^{(k)})\cup f(I_{i}^{(k)}, J_{j+1}^{(k)})$ is an interval).
 By the definition of Moran set, the basic intervals $ J_{j}^{(k)}$ can be ``random". That means its location can be changed  if the four Moran conditions are satisfied. Therefore, to prove
  $$f(I_{i}^{(k)}, J_{j}^{(k)})\cup f(I_{i}^{(k)}, J_{j+1}^{(k)})$$ is an interval, it suffices to prove one extreme case, which is shown in the first graph of Figure 1.  In other words, it remains to prove $f(P_1)\geq f(P_2)$.
In  the first graph of Figure 1,  the extreme case is shown in  the middle column (which is associated with the $i$-th column), we let the first $n_k-1$ intervals intersecting with each other.  The length of  these intersections is $\kappa c_1c_2\cdots c_k$, i.e. $$\dfrac{|J_{j}^{(k)}\cap J_{j+1}^{(k)}|}{|J_{j}^{(k)}|}=\kappa, 1\leq j\leq n_k-2.$$ Therefore, the length of
  $\cup_{j=1}^{n_k-1}J_{j}^{(k)}$ is $$(1+(n_k-2)(1-\kappa))c_1c_2\cdots c_k.$$
  Recall that $I=[a, a+t], J=[b, b+t]$.  We let
 \begin{eqnarray*}
P_1&=& (a+x_0, b+(1+(n_k-2)(1-\kappa))c_1c_2\cdots c_k)\\
P_2&=& (a+x_0-c_1c_2\cdots c_k, b+c_1c_2\cdots c_{k-1}-c_1c_2\cdots c_k)\\
l_1&=&c_1c_2\cdots c_k\\
l_2&=&c_1\cdots c_{k-1}(c_k(1+(n_k-2)(1-\kappa))-(1-c_k))).
\end{eqnarray*}
By  Taylor's theorem, we have
 \begin{eqnarray*}
f(P_1)-f( P_2)&=&l_1 \partial_x f(x,y) + l_2 \partial_y f(x,y)+\dfrac{1}{2}(l_1\partial_x+l_2\partial_y)^2f(\eta_1, \eta_2),
\end{eqnarray*}
 where  $(\eta_1, \eta_2)\in [0,1]\times [0,1].$
 Therefore, it suffices to prove
 $$l_1\partial_x f(x,y) + l_2 \partial_y f(x,y)\geq 0$$ and
 $$(l_1\partial_x+l_2\partial_y)^2f(x, y)\geq 0$$ for any $(x,y)\in [0,1]\times [0,1]$.
 However, these two inequalities are exactly the conditions given in theorem.

Similarly, we need to prove $(2)$. For this case, it suffices to prove the extreme case which is shown in the second graph of Figure 1.  More precisely, we need to show $f(P_3)\geq f(P_4),$ where
\begin{eqnarray*}
P_3&=& (a+(1+(n_k-2)(1-\kappa))c_1c_2\cdots c_k, b+c_1c_2\cdots c_{k-1})\\
P_4&=& (a+c_1c_2\cdots c_{k-1}-c_1c_2\cdots c_k, b)\\
l_3&=&c_1c_2\cdots c_{k-1}((2+(n_k-2)(1-\kappa)) c_k-1)\\
l_4&=&c_1\cdots c_{k-1}.
\end{eqnarray*}
In terms of Taylor's theorem again, we have
  \begin{eqnarray*}
f(P_3)-f( P_4)&=&l_3 \partial_x f(x,y) + l_4 \partial_y f(x,y)+\dfrac{1}{2}(l_3\partial_x+l_4\partial_y)^2f(\eta_3, \eta_4),
\end{eqnarray*}
where $(\eta_3, \eta_4)\in [0,1]^2.$
By the conditions in the theorem, we have $$l_3 \partial_x f(x,y) + l_4 \partial_y f(x,y)\geq 0, (l_3\partial_x+l_4\partial_y)^2f(x,y)\geq 0$$ for any $(x,y)\in [0,1]\times [0,1].$

For the other three cases, the proof is similar. We only need to prove the extreme cases which are shown in Figures $1$ and $2$.  For instance, when 
$$\partial_x f(x,y)<0, \partial_y f(x,y)<0, \mbox{ (see the third and fourth graphs in Figure 1)}$$
we need to prove 
$$f(P_5)\geq f(P_6),f(P_7)\geq f(P_8),$$ where 
 \begin{eqnarray*}
P_5&=& (a+x_0, b+c_1c_2\cdots c_{k-1}-c_1c_2\cdots c_{k-1}c_k)\\
P_6&=& (a+x_0+c_1c_2\cdots c_k, b+(1+(n_k-2)(1-\kappa))c_1c_2\cdots c_k)\\
l_5&=&-c_1c_2\cdots c_k\\
l_6&=&-c_1\cdots c_{k-1}(c_k(2+(n_k-2)(1-\kappa))-1)\\
P_7&=&(a+c_1c_2\cdots c_{k-1}-c_1c_2\cdots c_k,b)\\
P_8&=&(a+(1+(n_k-2)(1-\kappa))c_1\cdots c_k, b+c_1c_2\cdots c_{k-1})\\
l_7&=&c_1c_2\cdots c_{k-1}(1-c_k(2+(n_k-2)(1-\kappa)))\\
l_8&=&-c_1c_2\cdots c_{k-1}.
\end{eqnarray*}
Applying Taylor's theorem, we have 
  \begin{eqnarray*}
f(P_5)-f( P_6)&=&l_5 \partial_x f(x,y) + l_6 \partial_y f(x,y)+\dfrac{1}{2}(l_5\partial_x+l_6\partial_y)^2f(\eta_5, \eta_6)\\
f(P_7)-f( P_8)&=&l_7 \partial_x f(x,y) + l_8 \partial_y f(x,y)+\dfrac{1}{2}(l_7\partial_x+l_8\partial_y)^2f(\eta_7, \eta_8),
\end{eqnarray*}
where $(\eta_5, \eta_6), (\eta_7, \eta_8)\in [0,1]\times [0,1]$.
In terms of the associated conditions in Theorem \ref{Main1}, we have that for any $(x,y)\in[0,1]^2$,
  \begin{eqnarray*}
l_5 \partial_x f(x,y) + l_6 \partial_y f(x,y)&\geq &0\\
(l_5\partial_x+l_6\partial_y)^2f(x,y)&\geq &0\\
l_7 \partial_x f(x,y) + l_8 \partial_y f(x,y)&\geq &0\\
(l_7\partial_x+l_8\partial_y)^2f(x,y)&\geq &0.
\end{eqnarray*}
If $$\partial_x f(x,y)<0, \partial_y f(x,y)>0,$$ then by the first and second graphs in Figure 2, we only need to prove 
$$f(P_9)\geq f(P_{10}),f(P_{11})\geq f(P_{12}),$$ where 
 \begin{eqnarray*}
P_9&=&(a+x_0,b+(1+(n_k-2)(1-\kappa))c_1\cdots c_k)\\
P_{10}&=&(a+x_0+c_1c_2\cdots c_k,b+c_1c_2\cdots c_{k-1}-c_1c_2\cdots c_k)\\
l_9&=&-c_1c_2\cdots c_k\\
l_{10}&=&c_1c_2\cdots c_{k-1}(c_k(2+(n_k-2)(1-\kappa))-1)\\
P_{11}&=&(a+c_1c_2\cdots c_{k-1}-c_1c_2\cdots c_k, b+c_1c_2\cdots c_{k-1})\\
P_{12}&=&(a+(1+(n_k-2)(1-\kappa))c_1\cdots c_k,b)\\
l_{11}&=&-c_1c_2\cdots c_{k-1}(c_k(2+(n_k-2)(1-\kappa))-1)\\
l_{12}&=&c_1c_2\cdots c_{k-1}.
\end{eqnarray*}
By  Taylor's theorem and the conditions in theorem, we have 
 \begin{eqnarray*}
f(P_{9})-f( P_{10})&=&l_{9} \partial_x f(x,y) + l_{10} \partial_y f(x,y)+\dfrac{1}{2}(l_{9}\partial_x+l_{10}\partial_y)^2f(\eta_{9}, \eta_{10})\geq 0,\\
f(P_{11})-f( P_{12})&=&l_{11} \partial_x f(x,y) + l_{12} \partial_y f(x,y)+\dfrac{1}{2}(l_{11}\partial_x+l_{12}\partial_y)^2f(\eta_{11}, \eta_{12})\geq 0,
\end{eqnarray*}
where $(\eta_{9}, \eta_{10}), (\eta_{11}, \eta_{12})\in [0,1]^2.
$ 

For the final case, i.e. $$\partial_x f(x,y)>0, \partial_y f(x,y)<0$$ for any $(x,y)\in [0,1]^2$,  by the third and fourth graphs in Figure 2, we need to show
$$f(P_{13})\geq f(P_{14}),f(P_{15})\geq f(P_{16}),$$ where 
 \begin{eqnarray*}
P_{13}&=&(a+x_0+c_1c_2\cdots c_k, b+c_1c_2\cdots c_{k-1}-c_1c_2\cdots c_k)\\
P_{14}&=&a+x_0, b+(1+(n_k-2)(1-\kappa))c_1\cdots c_k\\
l_{13}&=&c_1c_2\cdots c_k\\
l_{14}&=&c_1c_2\cdots c_{k-1}(1-c_k(2+(n_k-2)(1-\kappa)))\\
P_{15}&=&(a+(1+(n_k-2)(1-\kappa))c_1\cdots c_k, b)\\
P_{16}&=&(a+c_1c_2\cdots c_{k-1}-c_1c_2\cdots c_k, b+c_1c_2\cdots c_{k-1})\\
l_{15}&=&c_1c_2\cdots c_{k-1}(c_k(2+(n_k-2)(1-\kappa))-1)\\
l_{16}&=&-c_1c_2\cdots c_{k-1}.
\end{eqnarray*}
  \begin{eqnarray*}
f(P_{13})-f( P_{14})&=&l_{13} \partial_x f(x,y) + l_{14} \partial_y f(x,y)+\dfrac{1}{2}(l_{13}\partial_x+l_{14}\partial_y)^2f(\eta_{13}, \eta_{14}),\\
f(P_{15})-f( P_{16})&=&l_{15} \partial_x f(x,y) + l_{16} \partial_y f(x,y)+\dfrac{1}{2}(l_{15}\partial_x+l_{16}\partial_y)^2f(\eta_{15}, \eta_{16}),
\end{eqnarray*}
where $(\eta_{13}, \eta_{14}), (\eta_{15}, \eta_{16})\in [0,1]^2.
$ 
By the conditions in theorem, we clearly have 
$$f(P_{13})-f( P_{14})\geq 0, f(P_{15})-f( P_{16})\geq 0.$$
\end{proof}
\section{Proof of Theorem \ref{SSS1}}
\begin{proof}[\bf{Proof of Theorem \ref{SSS1}}]
The proof is similar to Theorem \ref{Main1}. We only prove the case $$\partial_x f(x,y)>0, \partial_y f(x,y)>0, \mbox{ for any }(x,y)\in [0,1]^2.$$ 
Let $$I=f_{i_1\cdots i_{k-1}}([0,1])\subset C_{k-1}, J=g_{j_1\cdots j_{k-1}}([0,1])\subset D_{k-1}$$ be two basic intervals. By definition, we have 
$$\tilde{I}=\cup_{i=1}^{n}f_{i_1\cdots i_{k-1}i}([0,1]), \tilde{J}=\cup_{j=1}^{m}g_{j_1\cdots j_{k-1}j}([0,1]).$$  Denote $I_i=f_{i_1\cdots i_{k-1}i}([0,1]), J_j=g_{j_1\cdots j_{k-1}j}([0,1]), 1\leq i\leq n, 1\leq j\leq m.$

By Lemmas \ref{key1} and \ref{key2}, if we  prove  $$f(\tilde{I}, \tilde{J})=\cup_{i=1}^{n}\cup_{j=1}^{m}f(I_i, J_j)=f(I,J),$$ then we finish the proof of Theorem \ref{SSS1}. However,  for this statement it is enough to show the following:
\begin{itemize}
\item[(1)] for any $1\leq i\leq n$, we have $$f(I_i, \tilde{J})=f(I_i,\cup_{j=1}^{m}J_j )$$ is a closed interval; 
\item[(2)]$$f(I_i, \cup_{j=1}^{m}J_j )=f(I_{i+1},\cup_{j=1}^{m}J_j )$$ is a closed interval for any $1\leq i\leq n-1$. 
\end{itemize}
For  $(1)$, it only remains to prove that $$f(I_i, J_j )\cup f(I_i, J_{j+1})$$ is a closed interval.  Since $\partial_x f(x,y)>0, \partial_y f(x,y)>0$ for any $(x,y)\in [0,1]^2,$ we only need to prove the right endpoint of $f(I_i, J_j)$ is not less than the left endpoint of $f(I_i, J_{j+1})$.
Let $$P_{17}=(f_{i_1\cdots i_{k-1}i}(1), g_{j_1\cdots j_{k-1}}(0)+\lambda^{k-1}g_j(1))$$ be the right endpoint of $f(I_i, J_j)$, and 
$$P_{18}=(f_{i_1\cdots i_{k-1}i}(1)-\lambda^{k}, g_{j_1\cdots j_{k-1}}(0)+\lambda^{k-1}g_{j+1}(0))$$ be the left endpoint of $f(I_i, J_{j+1})$. 
 By  Taylor's theorem, we have 
 \begin{eqnarray*} 
f(P_{17})-f( P_{18})&=&l_{17} \partial_x f(x,y) + l_{18} \partial_y f(x,y)+\dfrac{1}{2}(l_{17}\partial_x+l_{18}\partial_y)^2f(\eta_{17}, \eta_{18}),
\end{eqnarray*}
where $l_{17}=\lambda^k, l_{18}=\lambda^{k-1}(g_j(1)-g_{j+1}(0)), $
and $(\eta_{17}, \eta_{18})\in [0,1]^2.$
By the assumptions $$\partial_x f(x,y)>0, \partial_y f(x,y)>0$$ for any $(x,y)\in [0,1]^2,$ if $ g_j(1)-g_{j+1}(0)\geq 0$, then we have 
$$l_{17} \partial_x f(x,y) + l_{18} \partial_y f(x,y)\geq 0.$$ If $j\in F_2,$ then the condition $$\dfrac{\partial_y f(x,y)}{\partial_x f(x,y)}\leq\dfrac{\lambda}{\tau_2}
$$ implies that 
$$\dfrac{\partial_y f(x,y)}{\partial_x f(x,y)}\leq\dfrac{\lambda}{\tau_2}\leq \dfrac{\lambda}{g_{j+1}(0)-g_j(1)}. 
$$ 
Therefore, under the conditions in the theorem, we always have 
$$l_{17} \partial_x f(x,y) + l_{18} \partial_y f(x,y)\geq 0.$$
It is easy to see that $$(l_{17}\partial_x+l_{18}\partial_y)^2f(x,y)\geq 0$$ if and only if 
$$(\lambda\partial_x+(g_j(1)-g_{j+1}(0))\partial_y)^2f(x,y)\geq 0.$$

Now, we prove $(2)$. It suffices to prove  the right endpoint of $f(I_i, \cup_{j=1}^{m}J_j)$ is not less than the left endpoint of $f(I_{i+1}, \cup_{j=1}^{m}J_j)$, i.e. 
$$f(P_{19})\geq f(P_{20}),$$
where $$P_{19}=(f_{i_1\cdots i_{k-1}}(0)+\lambda^{k-1}f_i(1), g_{j_1\cdots j_{k-1}}(0)+\lambda^{k-1})$$ is the right endpoint of $f(I_i, \cup_{j=1}^{m}J_j)$, and 
$$P_{20}=(f_{i_1\cdots i_{k-1}}(0)+\lambda^{k-1}f_{i+1}(0), g_{j_1\cdots j_{k-1}}(0))$$
is the left endpoint of $f(I_{i+1}, \cup_{j=1}^{m}J_j).$
By  Taylor's theorem again, it follows that 
\begin{eqnarray*}
f(P_{19})-f( P_{20})&=&l_{19} \partial_x f(x,y) + l_{20} \partial_y f(x,y)+\dfrac{1}{2}(l_{19}\partial_x+l_{20}\partial_y)^2f(\eta_{19}, \eta_{20}),
\end{eqnarray*}
where $l_{19}=\lambda^{k-1}(f_i(1)-f_{i+1}(0)), l_{20}=\lambda^{k-1}, $
and $(\eta_{19}, \eta_{20})\in [0,1]^2.$
With a similar discussion as the first case, by the conditions
\begin{equation*}
   \left\lbrace\begin{array}{cc}
              ( (f_i(1)-f_{i+1}(0))\partial_{x}+\partial_{y})^2 f(x,y)\geq 0  \mbox{ for }1\leq i\leq  n-1\\
                \tau_1\leq\dfrac{\partial_y f(x,y)}{\partial_x f(x,y)}                                \end{array}\right.
\end{equation*}
we always have 
$$l_{19} \partial_x f(x,y) + l_{20} \partial_y f(x,y)\geq 0
$$
and $$(l_{19}\partial_x+l_{20}\partial_y)^2f(x,y)\geq 0$$ for any $(x,y)\in [0,1]^2.$
Hence, we have finished the proof of Theorem \ref{SSS1} for the first case. 
For other  three cases, the proof is similar. 
Suppose that $$\partial_x f(x,y)<0, \partial_y f(x,y)<0, \mbox{ for any }(x,y)\in [0,1]^2.$$ 
Then to prove   $(1)$ and $(2)$, we only need to prove 
$$f(P_{21})\geq f(P_{22}),f(P_{23})\geq f(P_{24}),$$
where 
 \begin{eqnarray*}
P_{21}&=&(f_{i_1\cdots i_{k-1}i}(0), g_{j_1\cdots j_{k-1}j+1}(0))\\
P_{22}&=&(f_{i_1\cdots i_{k-1}i}(0)+\lambda^k, g_{j_1\cdots j_{k-1}j}(1))\\
l_{21}&=&-\lambda^k\\
l_{22}&=&\lambda^{k-1}(g_{j+1}(0)-g_j(1))\\
P_{23}&=&(f_{i_1\cdots i_{k-1}i+1}(0), g_{j_1\cdots j_{k-1}}(0))\\
P_{24}&=&(f_{i_1\cdots i_{k-1}i}(1), g_{j_1\cdots j_{k-1}}(0)+\lambda^{k-1})\\
l_{23}&=&\lambda^{k-1}(f_{i+1}(0)-f_i(1))\\
l_{24}&=&-\lambda^{k-1}.
\end{eqnarray*}
In terms of the conditions in Theorem \ref{SSS1}, we have 
$$f(P_{21})\geq f(P_{22}),f(P_{23})\geq f(P_{24}).$$
Similarly, if 
$$\partial_x f(x,y)<0, \partial_y f(x,y)>0, \mbox{ for any }(x,y)\in [0,1]^2, $$
then to prove   $(1)$ and $(2)$, we only need to prove 
$$f(P_{25})\geq f(P_{26}),f(P_{27})\geq f(P_{28}),$$
where 
 \begin{eqnarray*}
P_{25}&=&(f_{i_1\cdots i_{k-1}i}(0), g_{j_1\cdots j_{k-1}j}(1))\\
P_{26}&=&(f_{i_1\cdots i_{k-1}i}(0)+\lambda^k, g_{j_1\cdots j_{k-1}j+1}(0))\\
l_{25}&=&-\lambda^{k}\\
l_{26}&=&\lambda^{k-1}(g_{j}(1)-g_{j+1}(0))\\
P_{27}&=&(f_{i_1\cdots i_{k-1}i+1}(0), g_{j_1\cdots j_{k-1}}(0)+\lambda^{k-1})\\
P_{28}&=&(f_{i_1\cdots i_{k-1}i}(1), g_{j_1\cdots j_{k-1}}(0))\\
l_{27}&=&\lambda^{k-1}(f_{i+1}(0)-f_i(1))\\
l_{28}&=&\lambda^{k-1}.
\end{eqnarray*}
By virtue of the conditions in Theorem \ref{SSS1}, we have 
$$f(P_{25})\geq f(P_{26}),f(P_{27})\geq f(P_{28}).$$
Finally, we prove the last case, i.e. 
$$\partial_x f(x,y)>0, \partial_y f(x,y)<0, \mbox{ for any }(x,y)\in [0,1]^2.$$
By means of the conditions in Theorem \ref{SSS1}, we have 
$$f(P_{29})\geq f(P_{30}),f(P_{31})\geq f(P_{32}),$$
where 
 \begin{eqnarray*}
P_{29}&=&(f_{i_1\cdots i_{k-1}i}(1), g_{j_1\cdots j_{k-1}j+1}(0))\\
P_{30}&=&(f_{i_1\cdots i_{k-1}i}(1)-\lambda^k, g_{j_1\cdots j_{k-1}j}(1))\\
l_{29}&=&\lambda^{k}\\
l_{30}&=&\lambda^{k-1}(g_{j+1}(0)-g_{j}(1))\\
P_{31}&=&(f_{i_1\cdots i_{k-1}i}(1), g_{j_1\cdots j_{k-1}}(0))\\
P_{32}&=&(f_{i_1\cdots i_{k-1}i+1}(0), g_{j_1\cdots j_{k-1}}(0)+\lambda^{k-1})\\
l_{31}&=&\lambda^{k-1}(f_{i}(1)-f_{i+1}(0))\\
l_{32}&=&-\lambda^{k-1}.
\end{eqnarray*}
Therefore, by the four cases discussed above we prove that under the conditions in Theorem \ref{SSS1}, we always have $f(I,J)=f(\widetilde{I}, \widetilde{J})$. Hence, by Lemmas \ref{key1} and \ref{key2}, we prove Theorem \ref{SSS1}.
\end{proof}
\section{Final remarks}
The conditions in Theorem \ref{Main1} can be modified. We may adjust the   second and third conditions on the second order partial derivatives. That is, if 
\begin{equation*}
   \left\lbrace\begin{array}{cc}
                (c_k\partial_{x}+ \delta_{xy}(c_k(2+(n_k-2)(1-\kappa))-1)\partial_{y})^2f(x,y)\leq 0\\
(\partial_{y}+ \delta_{xy}(c_k(2+(n_k-2)(1-\kappa))-1)\partial_{x})^2f(x,y)\leq 0,
                \end{array}\right.
\end{equation*}
then  we may give some stronger conditions on the first order partial derivatives. 
For some functions which are not differentiable on $\mathbb{R}^2$, for instance 
$$f(x,y)=\dfrac{1}{(x-0.5)^2}+\alpha x+\beta y+\gamma xy,$$ where $\alpha, \beta, \gamma$ are parameters,  we may still use the Corollaries \ref{Cantor1} and \ref{Cantor2} for the domain $$(x,y)\in ([0,1/3]\cup [2/3,1])\times ([0,1/3]\cup [2/3,1]).$$
Furthermore, we may  consider functions with multiple variables. 
We leave these considerations to the reader.  
  \section*{Acknowledgements}
  This work is
supported by K.C. Wong Magna Fund in Ningbo University.
This work is also supported by National Natural Science Foundation of China with No.  11701302, and by Zhejiang Provincial Natural Science Foundation of China with
No.LY20A010009. The authors are grateful to the referees' suggestions and comments. 
%\bibliographystyle{plain}
%\bibliography{OnUnivoquePointsForSelfSimilarSets}

\end{document}